\documentclass[11pt,reqno]{amsart}
\usepackage{amsmath,amsthm}

\newtheorem{theorem}{Theorem}[section]
\newtheorem{lemma}[theorem]{Lemma}

\newtheorem{corollary}[theorem]{Corollary}

\theoremstyle{definition}

\newtheorem{definition}[theorem]{Definition}

\newtheorem{remark}[theorem]{Remark}

\newtheorem{example}[theorem]{Example}

\numberwithin{equation}{section}

\begin{document}

\title{Degenerate Monge-Type Hypersurfaces}

\author{David N. Pham}
\address{School of Mathematical Sciences, Rochester Institute of Technology, Rochester, NY 14623}
\email{dpham90@gmail.com}

\subjclass[2010]{53C50, 53B30}

\keywords{lightlike geometry, Monge hypersurface, degenerate metric}



\begin{abstract}
In this note, we extend the notion of a Monge hypersurface from its roots in semi-Euclidean space to more general spaces.   For the degenerate case, the geometry of these structures is studied using the Bejancu-Duggal method of screen distributions.
\end{abstract}

\maketitle
\section{Introduction}
\noindent Semi-Riemannian geometry is a well established branch of mathematics.  By comparison, the theory of lightlike manifolds is still relatively new and less developed.  If $(\overline{M},\overline{g})$ is a semi-Riemannian manifold and $(M,g)$ is a semi-Riemannian submanifold of $(\overline{M},\overline{g})$, the key to relating the geometry on $M$ with that of $\overline{M}$ is the fact that the tangent bundle of $\overline{M}$ splits as
\begin{equation}
\nonumber
T\overline{M}|_{M}= TM\oplus NM,
\end{equation}
where $TM$ and $NM$ are respectively the tangent bundle and normal bundle of $M$.  In lightlike geometry, this decomposition is no longer possible since a degenerate metric produces a non-empty intersection between $TM$ and $NM$.  

Lightlike submanifolds arise naturally in semi-Riemannian geometry as well as physics.  In semi-Riemannian geometry, the metric tensor is indefinite.  Consequently, there is no assurance that the induced metric on any given submanifold will remain non-degenerate. In general relativity, lightlike submanifolds model various types of horizons \cite{AK} \cite{Sul1} \cite{Sul2}.       

To deal with the problems posed by lightlike submanifolds\footnote{For an alternate approach to lightlike geometry, see \cite{Ku} \cite{Ku1}.}, Bejancu and Duggal introduced the notion of \textit{screen distributions} in \cite{DS1}, which provides a direct sum decomposition of $T\overline{M}$ with certain nice properties.  With a choice of screen distribution, one can induce geometric objects on a lightlike submanifold in a manner which is analogous to what is done in the classical theory of submanifolds.

For any developing theory of mathematics, examples clearly play an important role in testing ideas, developing concepts, and shaping the overall theory.  For the field of lightlike geometry, a number of instructive examples have come in the form of a Monge hypersurface.   In addition to being a source of examples for the field, Monge hypersurfaces are also interesting geometric objects in their own right \cite{Cal} \cite{Che} \cite{DS} \cite{DS1}.   As defined in \cite{DS1} \cite{DS}, a Monge hypersurface lives in semi-Euclidean space, which places limitations on their geometry.  In this note, we extend the notion of a Monge hypersurface from its roots in semi-Euclidean space to more general spaces.  These new structures, which we call \textit{Monge-type hypersurfaces}, allow for more general geometries, and could, in time, be a source of new and interesting examples of lightlike hypersurfaces.  

The rest of the paper is organized as follows.  In section 2, we review the method of screen distributions introduced by Bejancu and Duggal in \cite{DS1}.  In section 3, we develop the basic theory of Monge-type hypersurfaces as it pertains to lightlike geometry.  Lastly, in section 4, we conclude the paper with some basic examples.     

\section{Preliminaries}
\noindent In this section, we review the Bejancu-Duggal approach to lightlike geometry \cite{DS} \cite{DS1}.  We begin with the following definition:
\begin{definition}
\label{LightlikeDefinition}
An $r$-lightlike (or degenerate) manifold $(M,g)$ is a smooth manifold $M$ with a degenerate metric $g$, which satisfies the following conditions
\begin{itemize}
\item[1.] the radical space
\begin{equation}
\mbox{Rad}~T_pM:=\{u\in T_pM~|~g(u,v)=0~\forall v\in T_p M\}
\end{equation}
has dimension $r>0$ for all $p\in M$
\item[2.] the distribution defined via $p\mapsto \mbox{Rad}~T_pM$ is smooth.
\end{itemize}
\end{definition}
\begin{definition}
Let $(\overline{M},\overline{g})$ be a semi-Riemannian manifold.  A submanifold $M$ of $\overline{M}$ is a \textit{lightlike submanifold} if $(M,g)$ is a lightlike manifold, where $g$ is the induced metric on $M$.
\end{definition}
\noindent Hence, the fibers of $\mbox{Rad}~TM$ are
\begin{equation}
\label{lightlike_submanifold1}
\mbox{Rad} ~T_p M= T_p M\cap T_pM^\perp
\end{equation}
where
\begin{equation}
T_pM^\perp:=\{u\in T_p\overline{M}~|~~\overline{g}(u,v)=0~\forall v\in T_p M\}.
\end{equation}
Since $\dim {Rad}~T_pM>0$, $T\overline{M}|_{M}$ does not decompose as the direct sum of $TM$ and $TM^\perp$.  Consequently, the classical Gauss-Weingarten formulas breakdown for lightlike submanifolds.  As a way to remedy this problem, Bejancu and Duggal \cite{DS1}  introduced the notion of \textit{screen distributions}.   We now review this approach for the special case when $(M,g)$ is a lightlike hypersurface, that is, a lightlike submanifold of codimension 1 in $(\overline{M},\overline{g})$. Notice that this implies that $TM^\perp$ is a line bundle and $TM^\perp\subset TM$.  By (\ref{lightlike_submanifold1}), we have 
\begin{equation}
\mbox{Rad}~TM=TM^\perp.
\end{equation}
A \textit{screen distribution} $S(TM)$ is any smooth vector bundle for which
\begin{equation}
\label{screen_decomp}
TM=S(TM)\oplus TM^\perp.
\end{equation}
Notice that (\ref{screen_decomp}) implies that $\overline{g}$ is non-degenerate on $S(TM)$.  The fundamental result of \cite{DS1} for the case of lightlike hypersurfaces can be stated as follows:
\begin{theorem}
\label{ExistenceTheorem}
Let $(\overline{M},\overline{g})$ be a semi-Riemannian manifold and let $(M,g)$ be a lightlike hypersurface of $\overline{M}$.  For each screen distribution $S(TM)$, there exists a unique line bundle $\mbox{tr}(TM)$ which satisfies the following conditions:
\begin{itemize}
\item[(i)] $T\overline{M}|_{M}=TM\oplus \mbox{tr}(TM)$ 
\item[(ii)] given a non-vanishing local section $\xi$ of $TM^\perp$ which is defined on a neighborhood $U$ of $p\in M$, there exists a unique, non-vanishing local section $N_\xi$ of $\mbox{tr}(TM)$ defined on a neighborhood $U'\subset U$ of $p$ such that
\begin{itemize}
\item[(a)] $\overline{g}(\xi,N_\xi)=1$
\item[(b)] $\overline{g}(N_\xi,N_\xi)=0$
\item[(c)] $\overline{g}(W,N_\xi)=0$ for all $W\in \Gamma(S(TM)|_{U'})$.
\end{itemize}
\end{itemize}
\end{theorem}
\noindent The line bundle $\mbox{tr}(TM)$ appearing in Theorem \ref{ExistenceTheorem} is called the \textit{lightlike transversal bundle}.  Explicitly, $\mbox{tr}(TM)$ is constructed as follows.  Let $F$ be any vector bundle for which
\begin{equation}
F\oplus TM^\perp = S(TM)^\perp.
\end{equation}
Notice that $F$ is necessarily a line bundle.  For all $p\in M$, choose a non-vanishing local section $\xi$ of $TM^\perp$ and a non-vanishing local section $V$ of $F$ which are both defined on a neighborhood $U$ of $p$.  Since $\overline{g}$ is non-degenerate on $S(TM)$, it follows that
\begin{equation}
T\overline{M}|_M=S(TM)\oplus S(TM)^\perp.
\end{equation} 
This implies 
\begin{equation}
\nonumber
\overline{g}(\xi,V)\neq 0
\end{equation}
 on $U$.   The local section $N_\xi$ in Theorem \ref{ExistenceTheorem} is then given by
\begin{equation}
\label{transversal1}
N_\xi:=\frac{1}{\overline{g}(\xi,V)}\left(V-\frac{\overline{g}(V,V)}{2\overline{g}(\xi,V)}\xi\right).
\end{equation}
It can be shown that the 1-dimensional space spanned by $N_\xi$ is independent of the choice of $\xi$ or the bundle $F$.  Hence, $N_\xi$ determines a rank 1 distribution, which, in turn, defines the line bundle $\mbox{tr}(TM)$.

Using the decomposition of Theorem \ref{ExistenceTheorem}, one obtains a modified version of the Gauss-Weingarten formulas for lightlike hypersurfaces:
\begin{align}
\label{lightlike_connection}
\overline{\nabla}_X Y&= \nabla_X Y+h(X,Y)\\
\label{lightlike_second_form}
\overline{\nabla}_X V&=-A_V X+\nabla^t_X V
\end{align}
for all $X,Y\in \Gamma(TM)$ and $V\in \Gamma(\mbox{tr}(TM))$,  where
\begin{itemize}
\item[(i)] $\overline{\nabla}$ is the Levi-Civita connection on $(\overline{M},\overline{g})$
\item[(ii)] $\nabla_X Y$ and  $A_V X$ belong to $\Gamma(TM)$
\item[(iii)] $h(X,Y)$ and $\nabla^t_X V$ belong to $\Gamma(\mbox{tr}(TM))$.
\end{itemize}
It follows from (\ref{lightlike_connection}) that $\nabla$ is a connection on $M$ and $h$ is a $\Gamma (\mbox{tr}(TM))$-valued $C^\infty(M)$-bilinear form.  In addition, a direct verification shows that $\nabla$ is torsion-free and $h$ is symmetric.    In (\ref{lightlike_second_form}), $A_V$ is a $C^\infty(M)$-linear operator on $\Gamma(TM)$ and $\nabla^t$ is a connection on $\mbox{tr}(TM)$.  As in classical submanifold theory, $h$ is called the \textit{second fundamental tensor}, and $A_V$ is the \textit{shape operator} of $M$ in $\overline{M}$.   Lastly, the second fundamental form $B_\xi$ associated with a local section $\xi$ of $TM^\perp$ is defined by
\begin{equation}
\label{secondFundamentalForm}
B_\xi(X,Y):=\overline{g}(\overline{\nabla}_X Y,\xi).
\end{equation}
It follows from this definition that 
\begin{equation}
h(X,Y)=B_\xi(X,Y)N_\xi.
\end{equation}

The primary shortcoming of this approach is that some (not all) of the induced geometric objects on $(M,g)$ are dependent on the choice of $S(TM)$.  Hence, the search for spaces with canonical or unique screen distributions has been an area of research  for this approach \cite{Dug0} \cite{Dug1} \cite{Dug2} \cite{Bej} \cite{BFL}.  Fortunately, this framework does contain objects which are independent of the choice of screen distribution.  Consequently, these objects provide the aforementioned theory with well defined invariants.  We conclude this section by recalling some of these invariants:
\begin{definition}
Let $(M,g)$ be a lightlike hypersurface with screen distribution $S(TM)$ and let $B_\xi$ denote the second fundamental form associated with a local section $\xi$ of $TM^\perp$.  Then $(M,g)$ is
\begin{itemize}
\item[(i)] totally geodesic if $B_\xi\equiv 0$
\item[(ii)] totally umbilical if $B_\xi=\rho g$ for some smooth function $\rho$
\item[(iii)] minimal if 
\begin{equation}
\nonumber
\sum_{i=1}^n \varepsilon_i B_\xi(E_i,E_i)=0,
\end{equation}
where $E_i,~i=1,\dots, n$ is any orthonormal local frame of $S(TM)$, and $\varepsilon_i:= g(E_i,E_i)\in \{-1,1\}$.
\end{itemize}
\end{definition}
\begin{remark}
Since $B_\xi$ is given by (\ref{secondFundamentalForm}), statements (i) and (ii) of the above definition are clearly independent of the choice of screen distribution.  Although not quite as apparent, statement (iii) of the above definition is independent of both the choice of orthonormal frame and the choice of screen distribution.  In addition, notice that if $\xi'$ is another local section of $TM^\perp$ (defined on the same open set as $\xi$), then
\begin{equation}
B_{\xi'} = \lambda B_\xi,
\end{equation}
for some smooth non-vanishing function $\lambda$.  Hence, (i)-(iii) are also independent of the choice of local section $\xi$.
\end{remark}

\section{Monge-Type Hypersurfaces}
\noindent We begin with the following definition:
\begin{definition}
\label{MongeType}
A \textit{Monge-type hypersurface} $(M,g)$ and its ambient space $(\overline{M},\overline{g})$ are generated by a triple $(\widehat{M},\widehat{g},F)$, where
\begin{itemize}
\item[(i)] $(\widehat{M},\widehat{g})$ is a semi-Riemannian manifold, and
\item[(ii)] $F: \widehat{M}\rightarrow \mathbb{R}$ is a smooth function.
\end{itemize}
The ambient space $(\overline{M},\overline{g})$ is the semi-Riemannian manifold defined by
\begin{align}
\nonumber
\overline{M}&:=\mathbb{R}\times \widehat{M}\\
\nonumber
\overline{g}&:=-dx^0\otimes dx^0+\pi^\ast \widehat{g},
\end{align}
where $\pi: \overline{M}\rightarrow \widehat{M}$ is the natural projection map and $x^0$ is the natural coordinate associated with the $\mathbb{R}$-component of $\overline{M}$.  $(M,g)$ is the hypersuface in $(\overline{M},\overline{g})$ defined by
\begin{align}
\nonumber
M&:=\{(t,p)\in \overline{M}~|~t=F(p)\}\\
\nonumber
g&:=i^\ast \overline{g},
\end{align}
 where $i: M\hookrightarrow \overline{M}$ is the inclusion map.  The triple $(\widehat{M},\widehat{g},F)$ is a \textit{Monge-type generator}.
 \end{definition}
\begin{remark}
Let $\mathbb{R}^{n+1}_{k-1}$ denote $(n+1)$-dimensional semi-Euclidean space with index $k-1$, that is, the space $\mathbb{R}^{n+1}$ with metric
\begin{equation}
\nonumber
\eta := -\sum_{i=1}^{k-1}dx^i\otimes dx^i+\sum_{j=k}^{n+1}dx^j\otimes dx^j.
\end{equation}
If $\widehat{M}$ in Definition \ref{MongeType} is an open submanifold of $\mathbb{R}^{n+1}_{k-1}$ (with the induced metric), then $(M,g)$ coincides with the definition given in \cite{DS1} for a Monge hypersurface.  
\end{remark}
\begin{definition}
A Monge-type generator is degenerate if its associated Monge-type hypersurface is degenerate.  
\end{definition}
\begin{theorem}
\label{MongeTypeLightlike}
Let $(\widehat{M},\widehat{g},F)$ be a Monge-type generator.  Then the associated Monge-type hypersurface is lightlike iff
$\widehat{g}(\widehat{\xi},\widehat{\xi})=1$, where $\widehat{\xi}$ is the gradient of $F$ with respect to $\widehat{g}$.
\end{theorem}
\begin{proof}
Let $(\overline{M},\overline{g})$ and $(M,g)$ be defined as in Definition \ref{MongeType}.  Set
\begin{align}
\label{LevelSetFunctionM}
G&: = F\circ \pi-x^0\\
\label{NormalVectorM}
\xi&:=\mbox{grad}_{\overline{g}}~G,
\end{align}
where $\pi: \overline{M}\rightarrow \widehat{M}$ is the projection map and $\mbox{grad}_{\overline{g}}~G$ is the gradient of $G$ with respect to $\overline{g}$. Since $M=G^{-1}(0)$, it follows that $\xi$ is normal to $M$.  Hence, $M$ is lightlike iff
\begin{equation}
\overline{g}(\xi,\xi)=0.
\end{equation}
Let $\widehat{\xi}_L$ denote the unique lift of $\widehat{\xi}$ to $\overline{M}$.  Then 
\begin{equation}
\xi=\frac{\partial}{\partial x^0}+\widehat{\xi}_L.
\end{equation}
Theorem \ref{MongeTypeLightlike} then follows from the fact that
\begin{align}
\overline{g}(\xi,\xi)=-1+\overline{g}(\widehat{\xi}_L,\widehat{\xi}_L)=-1+\widehat{g}(\widehat{\xi},\widehat{\xi}).
\end{align}
\end{proof}
\begin{remark}
\label{VectorFieldLift}
Let $(\widehat{M},\widehat{g},F)$ be a Monge-type generator and $(\overline{M},\overline{g})$ its associated ambient space.     If $X$ is any vector field on $\widehat{M}$, we will denote its unique lift to $\overline{M}$ by $X_L$, that is, if 
\begin{align}
\nonumber
\pi_1:& \overline{M}\rightarrow \mathbb{R}\\
\nonumber
\pi_2:&\overline{M}\rightarrow \widehat{M}
\end{align}
are the natural projection maps, then
\begin{align}
\nonumber
\pi_{1\ast} X_L = 0\\
\nonumber
\pi_{2\ast} X_L = X.
\end{align}
To simplify notation in some places, we will not distinguish between $X_L$ and $X$. 
\end{remark}
\noindent From the proof of Theorem \ref{MongeTypeLightlike}, we have the following:
\begin{corollary}
\label{CorNormalVecM}
Let $(\widehat{M},\widehat{g},F)$ be a Monge-type generator and let $(\overline{M},\overline{g})$ and $(M,g)$ be defined as in Definition \ref{MongeType}.    Then
\begin{equation}
\nonumber
\xi:=\frac{\partial}{\partial x^0}+\widehat{\xi}_L,
\end{equation}
is normal to $M$, where $\widehat{\xi}:=\mbox{grad}_{\widehat{g}}F$.
\end{corollary}
\begin{corollary}
Let $(\widehat{M},\widehat{g},F)$ be a Monge-type generator.  If the associated Monge-type hypersurface is degenerate in the sense of Definition \ref{LightlikeDefinition}, then $\widehat{M}$ cannot be compact.
\end{corollary}
\begin{proof}
Let $\widehat{\xi}:=\mbox{grad}_{\widehat{g}} F$.  If $\widehat{M}$ is compact, then $F$ must have a critical point somewhere on $\widehat{M}$.  Hence, $\widehat{\xi}$ must vanish at some point.  The statement of the corollary now follows from Theorem \ref{MongeTypeLightlike}.
\end{proof}
\begin{corollary}
Let $(M,g)$ be a Monge hypersurface with generator $(U,\widehat{\eta},F)$, that is, $U$ is an open submanifold of $\mathbb{R}^{n+1}_{k-1}$ and $\widehat{\eta}$ is the induced metric.  Then $(M,g)$ is lightlike iff
\begin{equation}
\nonumber
-\sum_{i=1}^{k-1}\left(\frac{\partial F}{\partial x^i}\right)^2+\sum_{j=k}^{n+1}\left(\frac{\partial F}{\partial x^j}\right)^2=1.
\end{equation}
\end{corollary}
\begin{proof}
Let $\widehat{\xi}$ denote the gradient of $F$ with respect to $\widehat{\eta}$.  Then
\begin{equation}
\widehat{\xi}=-\sum_{i=1}^{k-1}\frac{\partial F}{\partial x^i}\frac{\partial}{\partial x^i}+\sum_{j=k}^{n+1}\frac{\partial F}{\partial x^j}\frac{\partial}{\partial x^j}.
\end{equation}
For $v\in T_x\mathbb{R}^{n+1}\simeq \mathbb{R}^{n+1}$,
\begin{equation}
\widehat{\eta}(v,v):=-\sum_{i=1}^{k-1}(v^i)^2+\sum_{j=k}^{n+1}(v^j)^2.
\end{equation}
The corollary now follows from Theorem \ref{MongeTypeLightlike}.
\end{proof}
\noindent The following is an immediate consequence of Definition \ref{MongeType}:
\begin{lemma}
\label{LocalFrameM}
Let $(M,g)$ be an $(n+1)$-dimensional Monge-type hypersurface with generator $(\widehat{M},\widehat{g},F)$.  For $p\in \widehat{M}$ and $(U,x^i)$ a coordinate neighborhood of $p$,  the vector fields 
\begin{align}
\nonumber
e_i:=\frac{\partial F}{\partial x^i}\frac{\partial} {\partial x^0} + \frac{\partial}{ \partial x^i},~~i=1,\dots, n+1
\end{align}
make up a local frame on $M$ in a neighborhood of $(F(p),p)$.
\end{lemma}
\begin{theorem}
\label{2ndFundamentalFormThm}
Let $(\widehat{M},\widehat{g},F)$ be a generator with ambient space $(\overline{M},\overline{g})$ and Monge-type hypersurface $(M,g)$.  In addition, let $\overline{\nabla}$ be the Levi-Civita connection on $\overline{M}$ and let $\xi$ be defined as in Corollary \ref{CorNormalVecM}.  Then the second fundamental form $B_\xi$ of $M$ satisfies
\begin{equation}
B_\xi(X,Y)=-\mbox{Hess}(F)(\pi_\ast X,\pi_\ast Y),
\end{equation}
for all $X,Y\in \Gamma(TM)$, where $\mbox{Hess}(F)$ is the Hessian of $F$ in $(\widehat{M},\widehat{g})$.
\end{theorem}
\begin{proof}
Be definition,
\begin{equation}
B_\xi(X,Y):=\overline{g}(\overline{\nabla}_X Y, \xi)=-\overline{g}(Y,\overline{\nabla}_X \xi).
\end{equation}
Let $q\in M\subset \overline{M}$ and let $(U,x^i)$ be a coordinate neighborhood of $\pi(q)$ in $\widehat{M}$.  Since $B_\xi$ is $C^\infty(M)$-bilinear, it suffices to show that 
\begin{equation}
B_\xi(e_i,e_j)=-\mbox{Hess}(F)(\pi_\ast e_i,\pi_\ast e_j),
\end{equation}
where $\{e_i\}$ is the local frame on $M$ in Lemma \ref{LocalFrameM} associated with $(U,x^i)$.  Consider the coordinate system $(\mathbb{R}\times U,(x^0,x^i))$ of $q$ in $\overline{M}$.  In this coordinate system, the coefficients of $\overline{g}$ and $\widehat{g}$ are related via
\begin{align}
\label{MetricCoef1}
\overline{g}_{ij}&=\widehat{g}_{ij},~~i,j>0\\
\label{MetricCoef2}
\overline{g}_{0i}&=0,~~i>0\\
\label{MetricCoef3}
\overline{g}_{00}&=-1.
\end{align}  
Let $\widehat{\nabla}$ be the Levi-Civita connection on $(\widehat{M},\widehat{g})$.  Then (\ref{MetricCoef1})-(\ref{MetricCoef3}) implies
\begin{equation}
\overline{\nabla}_{e_i} e_j= \frac{\partial^2 F}{\partial x^i \partial x^j}\frac{\partial}{\partial x^0}+\left(\widehat{\nabla}_{\partial_i} \partial_j\right)_L,
\end{equation}
where 
\begin{equation}
\nonumber
\partial_i:=\frac{\partial}{\partial x^i}.
\end{equation}
Hence,
\begin{align}
\overline{g}\left(\overline{\nabla}_{e_i} e_j,\xi\right)&=\frac{\partial^2  F}{\partial x^i \partial x^j}\overline{g}(\partial_0,\xi)+\overline{g}\left(\left(\widehat{\nabla}_{\partial_i} \partial_j\right)_L,\xi\right)\\
&=-\frac{\partial^2  F}{\partial x^i  \partial x^j}+\widehat{g}\left(\widehat{\nabla}_{\partial_i} \partial_j,\widehat{\xi}\right)\\
&=-\frac{\partial^2  F}{\partial x^i  \partial x^j}+dF\left(\widehat{\nabla}_{\partial_i} \partial_j\right)\\
&=-\mbox{Hess}(F)(\partial_i,\partial_j)\\
&=-\mbox{Hess}(F)(\pi_\ast e_i,\pi_\ast e_j)
\end{align}
where the third to last equality follows from the definition of $\widehat{\xi}$ in Theorem \ref{MongeTypeLightlike}.  
\end{proof}
\begin{corollary}
Let $(M,g)$ be a Monge-type hypersurface with generator $(\widehat{M},\widehat{g},F)$ and ambient space $(\overline{M},\overline{g})$.  Then $M$ is totally geodesic iff $\mbox{Hess}(F)\equiv 0$ on $\widehat{M}$.
\end{corollary}
\noindent In terms of its generator, the following gives a necessary and sufficient condition for a Monge-type hypersurface to be totally umbilical.
\begin{theorem}
\label{HessUmbilicalForm}
Let $(M,g)$ be a Monge-type hypersurface with generator $(\widehat{M},\widehat{g},F)$ and ambient space $(\overline{M},\overline{g})$.  Then $M$ is totally umbilical in $\overline{M}$ iff, for all $p\in \widehat{M}$, there exists a neighborhood $\widehat{U}$ of $p$ and a smooth function $\widehat{\rho}\in C^\infty(\widehat{U})$ such that 
\begin{equation}
\nonumber
\mbox{Hess}(F)=\widehat{\rho} \left(dF\otimes dF-\widehat{g}\right)
\end{equation}
on $\widehat{U}$.
\end{theorem}
\begin{proof}
Let $\xi$ be defined as in Corollary \ref{CorNormalVecM}, let $q\in M$ be any point, let $(U,x^i)$ be a coordinate neighborhood of $\pi(q)$ in $\widehat{M}$, where $\pi$ is the projection map from $\overline{M}$ to $\widehat{M}$, and let  $\{e_i\}$ be the local frame on $M$ from Lemma \ref{LocalFrameM} associated with $(U,x^i)$.  From Theorem \ref{2ndFundamentalFormThm}, the condition that $M$ be totally umbilical is equivalent to
\begin{equation}
\label{HessTotallyUmbilical1}
\mbox{Hess}(F)(\pi_\ast e_i, \pi_\ast e_j) =-\rho g(e_i,e_j)
\end{equation}
for some smooth function $\rho$ defined on a neighborhood 
\begin{equation}
V\subset \pi^{-1}(U)\cap M
\end{equation}
of $q$ in $M$.  Let $\widehat{\rho}$ be the smooth function on the open set $\pi(V)\subset \widehat{M}$ defined by
\begin{equation}
\rho=\widehat{\rho}\circ \pi.
\end{equation}
Expanding the right side of (\ref{HessTotallyUmbilical1}) gives 
\begin{align}
-\rho g(e_i,e_j)&= -\rho \overline{g}(e_i,e_j)\\
\label{MMhatRelation}
-\rho g(e_i,e_j)&=-\widehat{\rho}\left(-\frac{\partial F}{\partial x^i}\frac{\partial F}{\partial x^j}+\widehat{g}(\partial_i,\partial_j)\right)\\
\label{MMhatRelation}
-\rho g(e_i,e_j)&=\widehat{\rho}\left(\frac{\partial F}{\partial x^i}\frac{\partial F}{\partial x^j}-\widehat{g}(\partial_i,\partial_j)\right)
\end{align}
where it is understood that if the left side of (\ref{MMhatRelation}) is evaluated at $p\in V$, the right side is evaluated at $\pi(p)$.  Since $\pi_\ast e_i=\partial_i$ and 
\begin{equation}
dF(\partial_i)= \frac{\partial F}{\partial x^i},
\end{equation}
we have
\begin{equation}
\mbox{Hess}(F)=\widehat{\rho}\left(dF\otimes dF -\widehat{g}\right)
\end{equation}
on $\pi(V)$.  This completes the proof.
 \end{proof}

\begin{corollary}
\label{MongeHyperUmbilical}
Let $M$ be a Monge hypersurface of $\mathbb{R}^{n+2}_k$ with generator $(\widehat{U},\widehat{\eta},F)$.  Then $M$ is totally umbilical iff  there exists a smooth function $\widehat{\rho}$ on $\widehat{U}$ such that
\begin{equation}
\label{umbilical1}
\frac{\partial^2 F}{\partial x^i\partial x^j} =\widehat{\rho}\left(\frac{\partial F}{\partial x^i}\frac{\partial F}{\partial x^j}-\widehat{\eta}_{ij}\right),~1\le i,j\le n+1.
\end{equation}
\end{corollary}
\begin{theorem}
\label{MongeTypeCanonicalScreen}
Every degenerate Monge-type hypersurface has a canonical screen distribution, which is integrable.  If $(\widehat{M},\widehat{g},F)$ is a lightlike generator with ambient space $(\overline{M},\overline{g})$, the lightlike transversal line bundle associated with the canonical screen is spanned by
\begin{equation}
\nonumber
N_\xi=-\frac{1}{2}\left(\frac{\partial}{\partial x^0}-\widehat{\xi}_L\right),
\end{equation}
where $\widehat{\xi}:=\mbox{grad}_{\widehat{g}}F$.  In addition, the vector field $N_\xi$ satisfies $\overline{g}(\xi,N_\xi)=1$.
\end{theorem}
\begin{proof}
Let $(M,g)$ be a degenerate Monge-type hypersurface with generator $(\widehat{M},\widehat{g},F)$ and ambient space $(\overline{M},\overline{g})$.  Let
\begin{align}
V:=-\frac{\partial}{\partial x^0}\in \Gamma (T\overline{M})
\end{align}
and let $\xi$ be defined as in Corollary \ref{CorNormalVecM}.  The canonical screen distribution is then defined by setting
\begin{equation}
S(TM)=(\mathcal{L}_\xi\oplus \mathcal{L}_V)^\perp,
\end{equation}
where $\mathcal{L}_\xi$ and $\mathcal{L}_V$ are the line bundles over $M$ with sections $\xi|_M$ and $V|_M$ respectively.  Since
\begin{equation}
TM=\mathcal{L}_\xi^\perp,
\end{equation}
it follows that $S(TM)\subset TM$.  In addition, since 
\begin{equation}
\overline{g}(\xi,V)=1,
\end{equation}
it follows that $\overline{g}$ is non-degenerate on $S(TM)$.  Hence, $\xi_p\notin S(TM)_p$ for all $p\in M$.  This implies 
\begin{equation}
TM=S(TM)\oplus TM^\perp.
\end{equation}
From (\ref{transversal1}), the lightlike transversal line bundle $\mbox{tr}(TM)$ associated with $S(TM)$ is spanned by the vector field
\begin{equation}
N_\xi=-\frac{1}{2}\left(\frac{\partial}{\partial x^0}-\widehat{\xi}_L\right).
\end{equation} 
The fact that $\overline{g}(\xi,N_\xi)=1$ is a consequence of Theorem \ref{ExistenceTheorem}.

To show that $S(TM)$ is integrable, notice from the definition of $V$ and $\overline{g}$ that 
\begin{equation}
\overline{g}(W,V)=0\Longleftrightarrow Wx^0=0.
\end{equation}
for all $W\in \Gamma(T\overline{M})$.  Let $X,Y\in \Gamma (S(TM))$.   From the definition of $S(TM)$, it follows that
\begin{equation}
[X,Y]x^0=X(Yx^0)-Y(Xx^0)= X(0)-Y(0)=0.
\end{equation}
Hence $[X,Y]\in \mathcal{L}_V^\perp$.  Lastly, since $X,Y$ are vector fields on $M$, so is $[X,Y]$.  Hence, $[X,Y]\in \mathcal{L}_\xi^\perp$.  This completes the proof.
\end{proof}
\noindent Regarding minimal Monge-type hypersurfaces, we have the following result:
\begin{theorem}
Let $(\widehat{M},\widehat{g}, F)$ be a degenerate Monge-type generator.  The associated Monge-type hypersurface $(M,g)$ is minimal iff for all $\widehat{p}\in \widehat{M}$, there exists a neighborhood $\widehat{U}$ of $\widehat{p}$ and an orthonormal frame $\{\widehat{E}_i\}$ of the kernel of $dF|_{\widehat{U}}$ such that
\begin{equation}
\sum_{i=1}^n \varepsilon_i\mbox{Hess}(F)(\widehat{E}_i,\widehat{E}_i)=0,
\end{equation}
where $\varepsilon_i=\widehat{g}(\widehat{E}_i,\widehat{E}_i)\in \{-1,1\}$ and $n:=\dim\widehat{M}-1$.
\end{theorem}
\begin{proof}
Let $S(TM)$ be the canonical screen distribution on $(M,g)$ and let $(\overline{M},\overline{g})$ be the ambient space associated with $(\widehat{M},\widehat{g}, F)$.  In addition, let $\xi$ be the null vector field tangent to $M$ given by Corollary \ref{CorNormalVecM}.  

Suppose that $(M,g)$ is minimal.  By definition, this means that for all $p\in M$, there exists a neighborhood $U$ of $p$ and an orthonormal frame $\{E_i\}_{i=1}^n$ of $S(TM)|_U$ such that 
\begin{equation}
\sum_{i=1}^n \varepsilon_i B_\xi(E_i,E_i)=0,
\end{equation}
where $\varepsilon_i=g(E_i,E_i)$.  Let $\pi: \overline{M}\rightarrow \widehat{M}$ be the projection map.  If necessary, shrink $U$ so that the open set $\widehat{U}:=\pi(U)$ is covered by some coordinate system $(x^j)_{j=1}^{n+1}$.  Let $\{e_j\}_{j=1}^{n+1}$ be the local frame on $M$ associated with $(\widehat{U},x^j)$  (see Lemma \ref{LocalFrameM}).  Then
\begin{equation}
E_i = \sum_{j=1}^{n+1} \alpha^j_i e_j = \left(\sum_{j=1}^{n+1}\alpha^j_i\frac{\partial F}{\partial x^j}\right)\frac{\partial}{\partial x^0}+\sum_{j=1}^{n+1}\alpha^j_i \frac{\partial}{\partial x^j},\hspace*{0.1in} i=1,\dots, n
\end{equation}
for some smooth functions $\alpha^j_i$, $j=1,\dots, n+1$, $i=1,\dots, n$.  Since $E_i$ is a section of $S(TM)|_U$, we have $\overline{g}(E_i,\partial_0)=0$, which is equivalent to 
\begin{equation}
\label{kernel_dF}
\sum_{j=1}^{n+1}\alpha^j_i\frac{\partial F}{\partial x^j}=0.
\end{equation}
Set 
\begin{equation}
\widehat{E}_i = \sum_{j=1}^{n+1}\alpha^j_i \frac{\partial}{\partial x^j},\hspace*{0.1in} i=1,\dots, n
\end{equation}
where $\widehat{E}_i$ is regarded as a vector field on $\widehat{U}$.  As a consequence of (\ref{kernel_dF}), we have
\begin{equation}
\label{Ei_Lift}
E_i=(\widehat{E}_i)_L,
\end{equation}
where $(\widehat{E}_i)_L$ is the unique lift of $\widehat{E}_i$ to $U$.  In addition, notice that (\ref{kernel_dF}) is equivalent to
\begin{equation}
\label{kernel_dF2}
dF(\widehat{E}_i)=0,\hspace*{0.1in} i=1,\dots, n.
\end{equation}
Moreover,
\begin{equation}
g(E_i,E_j)=\widehat{g}(\widehat{E}_i,\widehat{E}_j)=\varepsilon_j \delta_{ij},
\end{equation}
where $\delta_{ij}=1$ if $i=j$ and zero otherwise.  Equation (\ref{kernel_dF2}) shows that $\widehat{E}_i$ belongs to the kernel of $dF$ on $\widehat{U}$.  By Theorem \ref{2ndFundamentalFormThm}, we have
\begin{align}
\nonumber
\sum_{i=1}^n \varepsilon_i \mbox{Hess}(F)(\widehat{E}_i,\widehat{E}_i)&=\sum_{i=1}^n \varepsilon_i \mbox{Hess}(F)(\pi_\ast (\widehat{E}_i)_L,\pi_\ast(\widehat{E}_i)_L)\\
\nonumber
&=-\sum_{i=1}^n \varepsilon_i B_\xi((\widehat{E}_i)_L,(\widehat{E}_i)_L)\\
\nonumber
&=-\sum_{i=1}^n \varepsilon_i B_\xi(E_i,E_i)\\
\nonumber
&=0.
\end{align}

For the converse, we use (\ref{Ei_Lift}) to define $E_i$ in terms of $\widehat{E_i}$.  This immediately guarantees that the $E_i$ are orthonormal and satisfy
\begin{align}
\nonumber
\overline{g}(E_i,\partial_0)&=0\\
\nonumber
\overline{g}(E_i,\xi)&=0,
\end{align}
where the last equality follows from the fact that the $\widehat{E_i}$ belong to the kernel of $dF$.  Hence, $\{E_i\}_{i=1}^n$ is a local orthonormal frame on $S(TM)$.  Running the above calculation in reverse shows that $(M,g)$ is minimal.  This completes the proof.
\end{proof}

\section{Examples}
\noindent In this section, two basic examples of degenerate Monge-type hypersurfaces are presented; both examples are totally umbilical. 
\begin{example}
Let  $(\widehat{M},\widehat{g},F)$ be the generator defined by
\begin{align}
\nonumber
\widehat{M}&:= \{x\in \mathbb{R}^{n+1}~|~x^{n+1}>0\}\\
\nonumber
\widehat{g}&:= \frac{1}{(x^{n+1})^2}(dx^1\otimes dx^1+\cdots + dx^{n+1}\otimes dx^{n+1})\\
\nonumber
F&:=\ln (x^{n+1}).
\end{align}
In other words, $(\widehat{M},\widehat{g})$ is $(n+1)$-dimensional hyperbolic space.  Let 
\begin{equation}
\widehat{\xi}:=\mbox{grad}_{\widehat{g}}F=x^{n+1}\frac{\partial}{\partial x^{n+1}}.
\end{equation}
Since $\widehat{g}(\widehat{\xi},\widehat{\xi})=1$, it follows from Theorem \ref{MongeTypeLightlike} that the associated Monge-type hypersurface is lightlike.   A direct verification shows that
\begin{equation}
\nonumber
\mbox{Hess}(F)(\partial_i,\partial_j)=
\left\{\begin{array}{cl}
-\frac{1}{(x^{n+1})^2} & i=j < n+1\\
0 & i=j= n+1\\
0 & i\neq j
\end{array}\right.
\end{equation}
It follows easily from this that
\begin{equation}
\nonumber
\mbox{Hess}(F)=dF\otimes dF-\widehat{g}.
\end{equation}
Hence, by Theorem \ref{HessUmbilicalForm}, the associated Monge-type hypersurface $(M,g)$ is totally umbilical.  The null vector field $\xi$ tangent to $M$ (see Corollary \ref{CorNormalVecM}) is
\begin{equation}
\xi=\frac{\partial}{\partial x^0} + x^{n+1}\frac{\partial}{\partial x^{n+1}}.
\end{equation} 
Let $\mbox{tr}(TM)$ denote the lightlike transversal bundle associated with the canonical screen (see Theorem \ref{MongeTypeCanonicalScreen}), and let $N_\xi$ denote the unique section of $\mbox{tr}(TM)$ associated with $\xi$ .  The induced linear connection $\nabla$ on $(M,g)$ is then given by 
\begin{equation}
\overline{\nabla}_X Y = \nabla_X Y + g(X,Y)N_\xi
\end{equation}
for all $X,Y\in \Gamma (TM)$, where $\overline{\nabla}$ is the Levi-Civita connection on the ambient space $(\overline{M}, \overline{g})$.  (Note that the second fundamental form $B_\xi$ is precisely $g$ in this example.)
\end{example}

\begin{example}
Let $(\widehat{M},\widehat{g},F)$ be the generator\footnote{The Lorentz manifold $(\widehat{M},\widehat{g})$ is actually a 2-dimensional submanifold of the Schwartzchild spacetime (see pp. 149 \cite{HE}).} defined by 
\begin{align}
\nonumber
\widehat{M}&:=\{(t,r)~|~-\infty < t < \infty,~~r>R\}\\
\nonumber
\widehat{g}&:=-\left(1-\frac{R}{r}\right)dt\otimes dt+\left(1-\frac{R}{r}\right)^{-1} dr\otimes dr\\
\nonumber
F&:=\sqrt{r}\sqrt{r-R}+R\ln\left(\sqrt{r}+\sqrt{r-R}\right),
\end{align}
where $R>0$ is a constant.    A direct verification shows
\begin{equation}
\widehat{\xi}:=\mbox{grad}_{\widehat{g}}F=\sqrt{\frac{r-R}{r}}\frac{\partial}{\partial r}.
\end{equation}
This shows that $\widehat{g}(\widehat{\xi},\widehat{\xi})=1$.  By Theorem \ref{MongeTypeLightlike}, the associated Monge-type hypersurface $(M,g)$ is lightlike.  For the Hessian of $F$, we have
\begin{equation}
\nonumber
\mbox{Hess}(F)(\partial_i,\partial_j)=
\left\{\begin{array}{cl}
-\frac{R\sqrt{r-R}}{2r^{5/2}} & i=j=t \\
0 & \mbox{otherwise}
\end{array}\right.
\end{equation}
From this, we have
\begin{equation}
\nonumber
\mbox{Hess}(F)=\rho(dF\otimes dF - \widehat{g}),
\end{equation}
where 
\begin{equation}
\rho= -\frac{R}{2r^{3/2}\sqrt{r-R}}.
\end{equation}
Hence, $(M,g)$ is totally umbilical (in the ambient space $(\overline{M},\overline{g})$) by Theorem \ref{HessUmbilicalForm}.  

The null vector field $\xi$ tangent to $M$ (see Corollary \ref{CorNormalVecM}) is
\begin{equation}
\nonumber
\xi:=\frac{\partial}{\partial x^0} + \sqrt{\frac{r-R}{r}}\frac{\partial}{\partial r}.
\end{equation}
Let $\mbox{tr}(TM)$ be the lightlike transversal bundle associated with the canonical screen on $M$ (see Theorem \ref{MongeTypeCanonicalScreen}).  The unique section of $\mbox{tr}(TM)$ associated with $\xi$ is then
\begin{equation}
\nonumber
N_\xi =-\frac{1}{2}\left(\frac{\partial}{\partial x^0}- \sqrt{\frac{r-R}{r}}\frac{\partial}{\partial r}\right).
\end{equation}
Using the above information, the induced linear connection on $(M,g)$ is then given by
\begin{equation}
\overline{\nabla}_{X} Y = \nabla_{X} Y + \rho g(X,Y)N_\xi,\hspace*{0.1in}\forall~X,Y\in \Gamma(TM).
\end{equation}
\end{example}

\end{document}